\documentclass[12pt,leqno]{amsart}
\usepackage{amssymb,amsthm,amsmath,latexsym,wasysym}
\usepackage{soul}
\usepackage[all]{xy}
\usepackage{xcolor}

\newcommand{\T}{{\sf T}}
\newcommand{\Tg}{{\sf T_{\bullet}}}

\newtheorem{theorem}{\sc Theorem}[section]
\newtheorem{thm}[theorem]{\sc Theorem}
\newtheorem{lem}[theorem]{\sc Lemma}
\newtheorem{prop}[theorem]{\sc Proposition}
\newtheorem{cor}[theorem]{\sc Corollary}
\newtheorem{conjecture}[theorem]{\sc Conjecture}

\newtheorem*{con1'}{Conjecture 1'}

\theoremstyle{definition}
\newtheorem{exe}[theorem]{\sc Example}
\newtheorem{rem}[theorem]{\sc Remark}

\newcommand{\go}[1]{o_{\bullet}(#1)}

\title[Generalized torsion elements]{Generalized torsion elements in groups}



\author[Bastos]{Raimundo Bastos}
\address[R. Bastos]{Departamento de Matem\'atica, Universidade de Bras\'ilia, Campus Universit\'{a}rio Darcy Ribeiro, 70910-900 Brasilia, DF, Brazil \\ bastos@mat.unb.br}
\author[Schneider]{Csaba Schneider}
\address[C. Schneider]{Departamento de Matem\'atica\\
Universidade Federal de Minas Gerais\\
31270-901 Belo Horizonte, MG, Brazil\\
csaba@mat.ufmg.br \\
 schcs.github.io/WP/}
\author[Silveira]{Danilo Silveira}
\address[D. Silveira]{Departamento de Ci\^encias Exatas e Aplicadas, Universidade Federal de Ouro Preto, 35931-008 
	Ouro Preto, MG,  Brazil \\ danilo.sancao@ufop.edu.br}

\begin{document}
\maketitle
\begin{abstract}
A group element is called a generalized torsion  if a finite product of its conjugates is equal to the identity. 
We prove that in a nilpotent or FC-group, the generalized torsion elements are all  torsion elements.
Moreover, we compute the generalized order of an element in a finite group $G$ using its character table. 
\end{abstract}

\maketitle

\section{Introduction}\label{sec:intro}

For a pair of elements $x$ and $y$ in a  group $G$, we write $x^y=y^{-1}xy$ for the conjugate of $x$ by $y$. The order of $x$, denoted by $o(x)$, is the least positive integer $k$ such that $x^k=1$; the order is infinite if no such $k$ exists. We say that $x$ is a torsion element if $o(x)$ is finite. The set of all torsion elements of $G$ will be denoted by $\T(G)$.  An element $x\in G$ is said to be a {\it generalized torsion} if there exist $g_1, \ldots, g_k\in G$ such that $$x^{g_1} x^{g_2} \ldots x^{g_k} =1.$$ 
We will denote by $\Tg(G)$ the set of all generalized torsion elements in $G$. The {\it generalized order} of $x \in \Tg(G)$, denoted by $o_{\bullet}(x)$, is defined to be the smallest  positive integer $n$ such $x^{g_1}\cdots x^{g_n}=1$ for some $g_1,\ldots,g_n\in G$.  Hence, the identity element, for example, has generalized order one. 
We say that $G$ has {\it generalized exponent} $k$, writing $\exp_{\bullet}(G)=k$, if $\Tg(G) = G$ and $k$ is the minimal positive integer such that all elements in $G$ have generalized order at most $k$.

Note that if $x$ is a torsion element of $G$, then $o_{\bullet}(x)\leq o(x)$. Thus  $\T(G) \subseteq \Tg (G)$. The reverse inclusion, however, 
does not hold. For example, in the infinite dihedral group $D_{\infty}$, we have 
\[\
T(D_{\infty}) = \{g \mid g^2=1\}\quad\mbox{while}\quad \Tg(D_{\infty}) = D_{\infty}.
\]
Moreover, there are  finitely generated torsion-free groups  where all elements are generalized torsions (see \cite[Problem 3.11]{KM},  Gorchakov \cite{Gorc} or Goryushkin \cite{Gor}).  Osin \cite[Corollary 1.2]{Osin} constructed an example of a torsion-free 2-generator group $G$ with exactly two  conjugacy classes (in particular, $\exp_{\bullet}(G)=2$).
More recently, generalized torsion elements in knot groups were studied in a number of papers (see \cite{IMM},  \cite{IMM23}, \cite{MT},  \cite{NR}
and references therein for further results).

In Section~\ref{sec:gen}, we focus on groups whose generalized torsions are torsions. We will prove that this holds in the class of FC-groups 
(that is, groups whose conjugacy classes are finite) as stated in 
the following theorem.

\begin{thm} \label{thm_FC}
If $G$ is an FC-group, then $\Tg(G) = \T (G)$.     
\end{thm}

In Section~\ref{sec:fin}, we adapt some known results of Arad, Stavi and Herzog~\cite{ASH} to obtain bounds for the generalized exponent of a finite group
in terms of the conjugacy classes. Based on results of~\cite{ASH}, we also present a practical method for calculating the generalized order for  groups whose character table is known.

In Section~\ref{sec:lcs}, we will prove the following theorem showing that certain powers of generalized torsion elements lie deep in the lower central series.

\begin{thm} \label{thm:nilpotent}
Let $x$ be an element of a group  $G$.
\begin{enumerate}
    \item If $o_{\bullet}(x)=k$, then $x^{k^{m}} \in \gamma_{m}(G)$, for any positive integer $m$.
    \item If $G$ is nilpotent, then $\Tg(G) = \T (G)$. 
\end{enumerate}
\end{thm}

\section{Generalized Torsion Elements}\label{sec:gen}

Observe that if $G$ is an abelian group, then $o_{\bullet}(g)=o(g)$, for all $g \in G$. However,  if $G$ is non-abelian, then $o_{\bullet}(g)$ need not be equal to $o(g)$. 
For example, in the symmetric group $S_n$, all elements are conjugate to their inverse and so $S_n$ has generalized exponent $2$ for all $n$. Thus, 
taking the $n$-cycle $\sigma=(12\cdots n)\in S_n$, we have that $o_{\bullet}(\sigma)=2$, while $o(\sigma)=n$, showing that an element of generalized order $2$ can have arbitrarily large order.

According to  Corollary \ref{cor:pdivides-gen-order}, if $G$ is a finite $p$-group with exponent $p$, then $o_{\bullet}(g) = o(g)$, for all $g\in G$. It is worth mentioning that if $G$ is finite, then $o_{\bullet}(x)$ does not need to divide the order of $G$ (see Example \ref{exe:Suzuki}). 
Nevertheless, the following result holds.

\begin{prop}
Let $G$ be a finite group. 
\begin{enumerate}
    \item If $G$ has an element with generalized order $2$, then $G$ has even order.
    \item $G$ can be embedded  into a finite group of generalized exponent $2$.
\end{enumerate}
\end{prop}
\begin{proof}
(1) Assume that a nontrivial element $x\in G$ has generalized order~2. By definition, there exists an element $g$ in $G$ such that $xx^g=1$. If $g\in Z(G)$, then $x$ has order 2 and so $G$ has even order. Thus, in what follows we may assume that $g\notin Z(G)$. The map 
\[
\rho_g:G\rightarrow G,\quad y\mapsto y^g
\]
is  permutation of the elements of $G$. Since $\rho_g(x)=x^{-1}$ and $\rho_g(x^{-1})=x$, we may write $\rho_g\in\mbox{Sym}(G)$ as a product of disjoint cycles and one 
of these cycles is $(x, x^{-1})$. Considering that disjoint cycles in $\mbox{Sym}(G)$ always commute, $\rho_g$ has even order. Now, the map 
\[
\rho: G\rightarrow \mbox{Aut}(G),\quad y\mapsto \rho_y
\]
is a homomorphism whose kernel coincides with $Z(G)$. Since $g\notin Z(G)$ and $\rho_g$ has even order,  $G$ has even order.

(2) Set $n=|G|$. It follows from Cayley's theorem \cite[1.6.8]{Rob} that $G$ can be embedded  into the symmetric group $S_n$. Further, as was observed 
before this result, $\exp_{\bullet}(S_n)=2$. 
\end{proof}

 Recall a group element $x\in G$ is said to be \emph{real} if $x^{-1} \in x^G$. In particular, a nontrivial element $x\in G$ has generalized order 2 if and only if there exists $g \in G$ such that $xx^g = 1$; that is, $x^{-1} = x^g \in x^G$. Thus,  $x$ is real if and only if $o_{\bullet}(x)=2$.

In the next result we collect some of the basic properties of generalized torsion elements.

\begin{prop} \label{prop_basic}
 Let $G$ be a group. 
\begin{enumerate}
    \item If $H$ is a subgroup of $G$, then $\Tg(H) \subseteq \Tg(G)$.
    \item If $K$ is a group and $\varphi \colon G \to K$ is a homomorphism, then $(\Tg(G))^{\varphi} \subseteq \Tg(K)$. 
    \item $\Tg(G)$ is a normal (and characteristic) subset of $G$.  
   \item If $x  \in \Tg(G) \cap Z(G)$, then $x \in \T (G)$. Moreover, if $G$ is abelian then $\Tg(G) = \T (G)$.  
        
\end{enumerate} 
\end{prop}
\begin{proof}
(1) If $x \in \Tg(H)$, then there exist $h_1, \ldots, h_k \in H$ such that $$1 = x^{h_1} \cdots x^{h_k}$$ and so, $x \in \Tg(G)$.

(2) If $x \in \Tg(G)$, then there exist $g_1, \ldots, g_r \in G$ such that $$ 1 = x^{g_1} \cdots x^{g_r}.$$ Since $\varphi$ is a homomorphism, it follows that $$ 1 = (x^{\varphi})^{g_1^{\varphi}} \cdots (x^{\varphi})^{g_r^{\varphi}} $$ and so, $x^{\varphi} \in \Tg(K)$. 

(3) Given an element $g\in G$, conjugation by $g$ induces a automorphism on $G$. Now, the result follows from the previous item. 

(4) If $x \in \Tg(G) \cap Z(G)$, then there exist $g_1, \ldots, g_r \in G$ such that $$ 1 = x^{g_1} \cdots x^{g_r} = \underbrace{x \cdots x}_{r \ \textrm{times} } = x^r.$$ In particular, $x \in \T (G)$. 
\end{proof}

We are now in the position to prove Theorem \ref{thm_FC}. 

\begin{proof}[Proof of Theorem $\ref{thm_FC}$]
It is clear that $\T (G) \subseteq \Tg(G)$. Choose arbitrarily an element $x \in \Tg(G)$. 

Then, there exist $g_1, \ldots, g_k \in G$ such that $x^{g_1} x^{g_2} \cdots x^{g_k} =1$. In particular, by construction, $x \in \Tg(H)$, where $H = \langle x,g_1, \ldots, g_k \rangle$. Since $G$ is an FC-group, so is $H$. Thus, all centralizers of $x, 
 g_1,\ldots, g_k$ in $H$ have finite index. Since the intersection of a finite set of subgroups each of which has finite index is itself of finite index \cite[1.3.12]{Rob},  center $Z(H)$, being the intersection of the centralizers of the generators of $H$, has finite index in $H$. 
 Therefore $H$ is central-by-finite. Set $n = |H:Z(H)|$. 
Define the map $\theta^*:H\to H$ as follows:
$$\begin{array}{cccc}
\theta^{\ast} \ : & \! H & \! \longrightarrow
& \! H \\
& \! h & \! \longmapsto
& \! h^n.
\end{array}$$      
By Schur's Theorem \cite[10.1.3]{Rob},  $\theta^*$ is an endomorphism of $H$.
By Proposition \ref{prop_basic}(3), $x^n = x^{\theta^{\ast}} \in \Tg(H)$. Since  $\mbox{Im}(\theta^{\ast}) \leqslant Z(H)$,  $x^n$ is a torsion element (Proposition \ref{prop_basic}(4)) and so, $x$ is also a torsion element.
\end{proof}

We obtain the following result as a corollary; this result is somewhat similar to Dietzmann's Lemma \cite[14.5.7]{Rob} that if $X\subseteq G$ is a finite normal set consisting of torsion elements, then  $\langle X \rangle$ is finite.

\begin{cor}
In a group $G$, a finite normal subset consisting of generalized torsion elements generates a finite normal subgroup.    
\end{cor}
\begin{proof} Let $X=\{ x_1,\ldots,x_k\} \subseteq \Tg(G)$  be a finite normal subset of $G$ and set $N=\left<X\right>$. Since $X$ is a normal set, 
all the conjugacy classes $x_i^N$ have at most $k$ elements and so the center $Z(N)=\bigcap_{i=1}^kC_N(x_i)$ has finite index.  
If $g\in N$, then $Z(N)\leq C_N(g)$, and hence $C_N(g)$ has finite index; that is, the conjugacy class $g^N$ is finite. Thereby, $N$ is an FC-group.

By Theorem \ref{thm_FC}, $X \subseteq \T (G)$. Thus, Dietzmann's Lemma implies that $N$ is a finite normal subgroup.  \end{proof}

\section{The generalized order in a finite group}\label{sec:fin}

In this section, we link the generalized order $\go g$ of an element $g\in G$ of a finite group  to the characters of $G$. This also provides a practical method for computing 
the generalized order in groups for which the irreducible characters are known.
Suppose in this section that $G$ is a finite group and let $C_1,C_2,\ldots,C_m$ be the conjugacy classes
of $G$ such that $C_1=\{1\}$. A conjugacy class $C$ is said to be {\em real} if $C^{-1}=C$, otherwise $C$ is {\em non-real}. Following~\cite{ASH}, we denote by $\lambda$
the number of real conjugacy classes distinct from $\{1\}$ and by $2\mu$ the  number of non-real conjugacy classes (which is always an even number). 
Hence the number $m$ of conjugacy classes 
of $G$ can be written as $m=1+\lambda+2\mu$.

The two assertions of the following proposition are proved in Lemmas~7.3 and 7.4 of~\cite{ASH}.

\begin{prop} \label{Thm:bounding}
Suppose that $G$ is a finite group and let $g\in G$.
\begin{enumerate}
\item $\go g$ is less than or equal to the number of conjugacy classes in $G$ that contain powers of $g$.
\item $\go g\leq 2\mu+2$.
\end{enumerate}
\end{prop}

The generalized order of an element $g$ of a finite group $G$ can be calculated using the character table of $G$. Suppose that 
$\mbox{Irr}(G)$ denotes the set of irreducible characters of $G$. 
For a conjugacy class $C\subseteq G$ and for $k\geq 1$, let $\alpha_{C,k}$ be the number of $k$-tuples $(g_1,\ldots,g_k)\in C^k$ such that
$g_1\cdots g_k=1$. That is, $\alpha_{C,k}$ counts how many ways the identity can be written as a product of elements in $C$. 
For $g\in C$, we have that  
\[
\go g=\min\{k\geq 1\mid \alpha_{C,k}>0\}.
\]
The following lemma appeared in~\cite[Lemma~10.10]{ASH}; see also~\cite[Equation~(1)]{Sha}.

\begin{thm}\label{th:char}
Using the notation in the previous paragraph, 
\begin{equation}\label{eq:alpha}
\alpha_{C,k}=\frac{|C|^k}{|G|}\sum_{\chi\in{\rm Irr}(G)}\frac{\chi(g)^k}{\chi(1)^{k-2}}.
\end{equation}
\end{thm}

Theorem~\ref{th:char} gives an computationally efficient method for calculating the generalized order for elements in finite groups whose character tables are known.

\begin{exe}
    Suppose that $G$ is the group number $3$ among the groups of order $18$ in GAP~\cite{GAP4}. 
    The group $G$ has 9 conjugacy classes and 9 irreducible representations. Suppose that $C$ is 
    the conjugacy class number~8 according to the numbering given by GAP. Then one can 
    compute, using~\eqref{eq:alpha}, that $\alpha_{C,1}=\alpha_{C,2}=0$, but $\alpha_{C,3}=243$. 
    Hence the identity element $1\in G$ can be written as a product $g_1g_2g_3$ with $g_i\in C$ in $243$ ways and in particular 
    $\go g=3$ for all $g\in C$. 
\end{exe}

\begin{exe} \label{exe:Suzuki}
    Suppose that $G$ is the Suzuki group $\mbox{Sz}(8)$ and assume that $C$ is the 
    conjugacy class number three in the numbering by GAP. 
    Using GAP, we computed that $\alpha_{C,1}=\alpha_{C,2}=0$, but $\alpha_{C,3}=196,560$. 
    Thus the identity element of $G$ can be written in $196,560$ ways as a product $g_1g_2g_3$ with 
    $g_i\in C$. In particular, $\go g=3$ for all $g\in C$. 
    Interestingly, $196,560$ coincides with the kissing number of the $24$-dimensional 
Leech lattice and is equal to the coefficient of the first non-constant  term of the modular form the the lattice; \cite[Section~2]{CS}.
\end{exe}

\begin{cor}
Let $G$ be a finite group. If $H$ is a core-free  subgroup of $G$, then $\exp_{\bullet}(G)\leq 2^{|G:H|-1}$.
\end{cor}

\begin{proof}
Let $R$ be the set of all right cosets of $H$. Every element $g$ in $G$ induces a permutation on $R$ by right multiplication $(Hx)g=H(xg)$. Since $H$ is core-free, 
$G$ gets embedded into the symmetric group $S_n$ where $n=|G:H|$. 
 An important result due to Liebeck and Pyber \cite[Theorem 2]{LP} states that the number of conjugacy classes of any subgroup of $S_n$ is at most $2^{n-1}$. Now, the result follows from Proposition \ref{Thm:bounding}(1).
\end{proof}

It is known that many  finite nonabelian  simple groups have generalized exponent less than or equal to $3$ (see \cite[Theorem 3]{VG} and \cite[Chapters 1 and 2]{ASH}). 
In \cite[Theorem 2.6]{Sha}, Shalev showed that if $G$ is a  finite nonabelian  simple group and  $x\in G$ be chosen at
random, then the probability that $(x^G)^3=G$ tends to 1 as $|G|\rightarrow \infty.$ These facts support the following conjecture.

\begin{conjecture}
The generalized exponent of a finite non-abelian simple group is at most $3$.
\end{conjecture}

\section{Relations between generalized torsion and the terms of the lower central series}\label{sec:lcs}

We define recursively \emph{commutators} of weight $1,2,\ldots$ in elements $x_1,x_2,\ldots$ of a group $G$ as follows. The elements $x_1,x_2,\ldots$ are commutators of weight 1, $[x_i,x_j]=x_i^{-1}x_j^{-1}x_ix_j$, with $i\neq j$, are commutators of weight 2 and if $c_1$ and $c_2$ are commutators of weight $w_1$ and $w_2$, respectively, then $[c_1,c_2]$ is a commutator of weight $w_1+w_2$. Here, $c_1$ and $c_2$ are called left and right \emph{sub-commutators}, respectively.  The \emph{first entry} in  
a commutator $[c_1,c_2]$ is defined as the first entry of $c_1$, while the first entry of 
a commutator $x$ of weight one is of course just $x$. 
In case brackets are omitted,  the commutators are assumed left-normed, for example, $[x_1,x_2,x_3] = [[x_1,x_2],x_3]$.
The terms $\gamma_i(G)$ of the lower central series of $G$ are defined recursively as $\gamma_1(G)=G$ and $\gamma_{i+1}(G)=[\gamma_i(G),G]$ for 
$i\geq 1$. In particular, $\gamma_2(G)=G'$ is the commutator (or derived) subgroup. It is well-known that $\gamma_i(G)$ is the subgroup of $G$ generated by all commutators of weight $i$ in the elements of $G$.

We quote the following well-known lemma (see \cite[Lemma 5.1.5 and Exercise 5.1.4]{Rob} and \cite[ChapterIII, Section~9.4]{HuppertI}). It will be used in the rest of the paper, often without explicit reference. 

\begin{lem}\label{prop_comm}
For elements $x, y, z$ of a group $G$ and a positive integer $k$, the following identities are valid:
\begin{enumerate}
\item $xy=yx[x,y]$
\item $x^y=x[x,y]$
    \item $[xy,z]=[x,z][x,z,y][y,z]$
    \item $[x^k,y]=[x,y]^{x^{k-1}}[x,y]^{x^{k-2}}\cdots [x,y]^x[x,y]$
   \item $x^ky^k=(xy)^kc_2^{\binom{k}{2}}\ldots c_i^{\binom{k}{i}}\ldots c_{k-1}^kc_k$, where $c_i\in \gamma_i(\left<x,y\right>)$ for each non-negative integer $i$.
\end{enumerate}
\end{lem}
The item (5) above is known as the Hall-Petrescu formula.

\begin{lem} \label{lem:commutators-powers}
Let $k\geq 2$, and let $x, g_1,\ldots,g_k$ be elements in a group $G$.
\begin{enumerate}
    \item We have that  
    \[
    x^{g_1}x^{g_2}\cdots x^{g_k}=x^k\sigma_2
    \]
where $\sigma_2$ is a product of commutators of weight at least $2$ and the element $x$ is the first entry of all the factors of $\sigma_2$.
    \item If $o_{\bullet}(x)=k$, then  $x^{k}=c_1\cdots c_r$, where each $c_i=[c_{i,1},c_{i,2}]$ is a commutator of weight at least $2$ such that the first entry of $c_{i,2}$ is $x$.
    \item If $o_{\bullet}(x)=k$ and $c_m$ is a commutator of weight $m$ with $x$ in some entry, then $c_m^k$ is a product of commutators of weight at least $m+1$  and $x$ appears in some entry of all factors of $c_m^k$.
    \item If $o_{\bullet}(x)=k$ and $\sigma_m$ is a product of commutators of weight at least $m$ with $x$ in some entry of all its factors, then $\sigma_m^k$ is a product of commutators of weight at least $m+1$  and $x$ appears in some entry of all its factors.
    \item If $o_{\bullet}(x)=k$, then $(x^{g_1}x^{g_2}\cdots x^{g_k})^{k^m}=x^{k^{m+1}}\sigma_{m+1}$, where $\sigma_{m+1}$ is a product of commutators of weight at least $m+1$ and the element $x$ appears in some entry of all factors of $\sigma_{m+1}$.
  \end{enumerate}
\end{lem}
\begin{proof}
(1) We proceed by induction on $k.$ If $k=2$, then 
$$x^{g_1}x^{g_2}=x[x,g_1]x[x,g_2]=x^2[x,g_1][x,g_1,x][x,g_2].$$ Assuming the result holds for $k\geq 2,$ we get 
$$(x^{g_1}x^{g_2}\cdots x^{g_k})x^{g_{k+1}}=(x^k\sigma_2)x[x,g_{k+1}]=x^{k+1}\sigma_2[\sigma_2,x][x,g_{k+1}],$$ where $\sigma_2$ is a product of commutators of weight at least $2$ and the element $x$ is the first entry of the factors of $\sigma_2$. 
Now, the result follows by applying the Lemma \ref{prop_comm}(3) several times to the commutator $[\sigma_2,x]$.

(2)  If $o_{\bullet}(x)=k$, then there exist elements $g_1,\ldots, g_k\in G$ such that   $1=x^{g_1}x^{g_2}\cdots x^{g_k}$. By previous item, we can write $$1=x^{g_1}x^{g_2}\cdots x^{g_k}=x^k\sigma_2,$$ 
where $\sigma_2$ is a product of commutators of weight at least $2$ and the element $x$ is the first entry of all the factors of $\sigma_2$. Thus, $x^k=\sigma_2^{-1}$.

(3)  We proceed by induction on $m$. The basic step $m=1$ follows by item (2). Assume the result holds for all positive integers up to $m$. Let $c_{m+1}$ be a commutator of weight $m+1$ with $x$ in some entry.  Write $c_{m+1}=[c_i,c_j]$ where $c_i, c_j$ are commutators of weight $i$ and $j$, respectively, and $i+j=m+1.$ Without loss of generality, we can assume that $x$ occurs in the left sub-commutator $c_i$. By Lemma \ref{prop_comm}(4) we get
\begin{align*}
   [c_i^k,c_j]& =    [c_i,c_j]^{c_i^{k-1}}[c_i,c_j]^{c_i^{k-2}}\cdots [c_i,c_j]^{c_i}[c_i,c_j]\\
    & =    ([c_i,c_j][c_i,c_j,c_i^{k-1}])
    \cdots ([c_i,c_j][c_i,c_j,c_i])[c_i,c_j]\\
    & =  [c_i,c_j]^k\sigma_{m+2}
\end{align*}
where $\sigma_{m+2}$ is a product of commutators of weight at least $m+2$ with $x$ in some entry of all factors. By the induction hypothesis, $c_i^k$ is a product of commutators of weight at least $i+1$  and $x$ appears in some entry of all its factors, say $c_i^k=c_{i,1}\cdots c_{i,r}$. Since 
\[
c_{m+1}^k=[c_i,c_j]^k=[c_i^k,c_j]\sigma_{m+2}^{-1},
\]
the result follows by applying several times Lemma \ref{prop_comm}(3) to $[c_i^k,c_j]=[c_{i,1}\cdots c_{i,r},c_j]$.

(4) We proceed by induction on the number $r$ of factors of $\sigma_m$. The basic step $r=1$ follows by the previous item. Assume that the result holds for all positive integers up to $r\geq 1$ and set $\sigma_m=\tau_1\cdots\tau_r\tau_{r+1}$ where each $\tau_i$ is a commutator of weight at least $m$ with $x$ in some entry. By Lemma \ref{prop_comm}(5) we can deduce that 
$$\sigma_m^k=(\tau_1\cdots\tau_r\tau_{r+1})^k=(\tau_1\cdots\tau_r)^k\tau_{r+1}^k\sigma_{m+1}$$
where $\sigma_{m+1}$ is a product of commutators of weight at least $m+1$ with $x$ appearing in some entry in each factor. Thus, the result follows applying the induction hypothesis to $(\tau_1\cdots\tau_r)^k$ and to $\tau_{r+1}^k$. 

(5)  We show item (5) by induction on $m$. Firstly, we will show the basic step $m=1$. By item (1), we can write 
$x^{g_1}x^{g_2}\cdots x^{g_k}=x^k\sigma_2$. 
By parts (1) and (5) of Lemma \ref{prop_comm}, we have  
$$(x^{g_1}x^{g_2}\cdots x^{g_k})^k=(x^k\sigma_2)^k=x^{k^2}\tilde{\sigma_2}$$
where $\tilde{\sigma_2}$ is a product of commutators of weight at least 2 and the elements $x$ appears in some entry of all factors.

 Now, assume that the result holds for  all positive integers up to $m\geq 1$. By the induction hypothesis, we  get  
$$(x^{g_1}x^{g_2}\cdots x^{g_k})^{k^{m+1}}=((x^{g_1}x^{g_2}\cdots x^{g_k})^{k^{m}})^k=(x^{k^{m+1}}\sigma_{m+1})^k$$ 
where $\sigma_{m+1}$ is a product of commutators of weight at least $m+1$ and the element $x$ appears in some entry of all factors. By Lemma \ref{prop_comm}(5), we can deduce that $(x^{k^{m+1}}\sigma_{m+1})^k=x^{k^{m+2}}\sigma_{m+1}^k\sigma_{m+2}$, where $\sigma_{m+2}$ is a product of commutators of weight at least $m+2$ and the element $x$ appears in some entry of all the factors. Thus, the result follows applying item (4) to $\sigma_{m+1}^k$.
\end{proof}

We are now ready to prove Theorem \ref{thm:nilpotent}.

\begin{proof}
(1) Let $x$ be an element in a group $G$ with generalized order $o_{\bullet}(x)=k$. Thus, there exist elements $g_1,\ldots,g_k$ in $G$ such that $x^{g_1}x^{g_2}\cdots x^{g_k}=1$. By Lemma \ref{lem:commutators-powers}~(5), we have, for every positive integer $m$, 
$$1=(x^{g_1}x^{g_2}\cdots x^{g_k})^{k^{m-1}}=x^{k^m}\sigma_m$$
 where $\sigma_{m}$ is a product of commutators of weight at least $m$. Thus, $x^{k^m}\in \gamma_m(G)$.

(2) Let $c$ be the nilpotency class of $G$. We need to show that $\Tg(G) \subseteq \T (G)$. Choose arbitrarily $x \in \Tg(G)$. By the previous item $$x^{k^{c+1}} \in \gamma_{c+1}(G)=\{1\}.$$ As $x \in \Tg(G)$ has been chosen arbitrarily, we conclude that $\Tg(G) \subseteq \T (G)$. 
\end{proof}

A group $G$ is said to be orderable, if there is a total order on $G$ such that $a\leq b$ implies that $xay\leq xby$ for all $a,b,x,y\in G$. 
It is known that torsion-free nilpotent groups are orderable; see~\cite{Bludov}.

\begin{rem}
If $G$ is an orderable group then $\Tg(G)=\{1\}$. It is known that the converse does not hold in general (see \cite{Bludov}). We can deduce from previous result that if $G$ is nilpotent with $\Tg(G)=\{1\}$, then $G$ is orderable.  
\end{rem}

\begin{cor}\label{cor:pdivides-gen-order}
If $x$ is an element in a nilpotent $p$-group $G$, then  $p$ divides $o_{\bullet}(x)$.  
\end{cor}

\begin{proof}
Let $c$ be the nilpotency class of $G$. Since $G$ is a $p$-group we get that $x$ is a generalized torsion element. Setting $o_{\bullet}(x)=k$, it follows from Theorem \ref{thm:nilpotent}(1) that $x^{k^{c+1}}=1$. Consequently, $p$ divides $k^{c+1}$ and so, $p$ divides $k$.   
\end{proof}

\begin{rem}
The previous result cannot be improved. In general, if $G$ is a $p$-group and $g\in G$, then $o_{\bullet}(g)$ need not be a $p$-power. Let $G$ be the $8$-th group of order $81$ from the GAP Small Groups Library. Then $G$ contains elements with generalized torsion order $6$. 
\end{rem}

\section*{Acknowledgements}
This work was partially supported by DPI/UnB and FAPDF (Brazil). 
The second author acknowledges the financial support of  the CNPq projects 
\textit{Produtividade em Pesquisa} (project no.: 308212/2019-3)  
and \textit{Universal} (project no.: 421624/2018-3 and 402934/2021-0) and the Fapemig Project \emph{Universal} (project no.: APQ-00971-22). 
The third author was partially supported by FAPEMIG  RED-00133-21.


\begin{thebibliography}{GAP22}

\bibitem[ASH85]{ASH}
Z.~Arad, J.~Stavi, and M.~Herzog.
\newblock Powers and products of conjugacy classes in groups.
\newblock In {\em Products of conjugacy classes in groups}, volume 1112 of {\em
  Lecture Notes in Math.}, pages 6--51. Springer, Berlin, 1985.

\bibitem[Blu72]{Bludov}
V.~V. Bludov.
\newblock An example of an unorderable group with strictly isolated identity.
\newblock {\em Algebra i Logika}, 11:619--632, 736, 1972.

\bibitem[CS99]{CS}
J.~H. Conway and N.~J.~A. Sloane.
\newblock {\em Sphere packings, lattices and groups}, volume 290 of {\em
  Grundlehren der mathematischen Wissenschaften}.
\newblock Springer-Verlag, New York, third edition, 1999.

\bibitem[GAP22]{GAP4}
The GAP~Group.
\newblock {\em {GAP -- Groups, Algorithms, and Programming, Version 4.12.2}}, gap-system.org, 
  2022.

\bibitem[Gor67]{Gorc}
Ju.~M. Gor\v{c}akov.
\newblock An example of a {$G$}-periodic torsion-free group.
\newblock {\em Algebra i Logika Sem.}, 6(3):5--7, 1967.

\bibitem[Gor73]{Gor}
A.~P. Gorju\v{s}kin.
\newblock An example of a finitely generated {$G$}-periodic torsion-free group.
\newblock {\em Sibirsk. Mat. \v{Z}.}, 14:204--207, 239, 1973.

\bibitem[Hup67]{HuppertI}
B.~Huppert.
\newblock {\em Endliche {G}ruppen {I}},
\newblock volume 134 of {\em Grundlehren der mathematischen Wissenschaften}, 
  Springer-Verlag, Berlin-New York, 1967.

\bibitem[IMT21]{IMM}
Tetsuya Ito, Kimihiko Motegi, and Masakazu Teragaito.
\newblock Generalized torsion and {D}ehn filling.
\newblock {\em Topology Appl.}, 301:Paper No. 107515, 14, 2021.

\bibitem[IMT23]{IMM23}
Tetsuya Ito, Kimihiko Motegi, and Masakazu Teragaito.
\newblock Generalized torsion for hyperbolic 3-manifold groups with arbitrary
  large rank.
\newblock {\em Bull. London Math. Soc.} (to appear), arxiv.org/abs/2112.00418, 2021.

\bibitem[KM22]{KM}
E.~I. Khukhro and V.~D. Mazurov.
\newblock Unsolved problems in group theory. {T}he {K}ourovka notebook, arxiv.org/1401.0300, revision 26, 2022.

\bibitem[LP97]{LP}
Martin~W. Liebeck and L\'{a}szl\'{o} Pyber.
\newblock Upper bounds for the number of conjugacy classes of a finite group.
\newblock {\em J. Algebra}, 198(2):538--562, 1997.

\bibitem[MT17]{MT}
Kimihiko Motegi and Masakazu Teragaito.
\newblock Generalized torsion elements and bi-orderability of 3-manifold
  groups.
\newblock {\em Canad. Math. Bull.}, 60(4):830--844, 2017.

\bibitem[NR16]{NR}
Geoff Naylor and Dale Rolfsen.
\newblock Generalized torsion in knot groups.
\newblock {\em Canad. Math. Bull.}, 59(1):182--189, 2016.

\bibitem[Osi10]{Osin}
Denis Osin.
\newblock Small cancellations over relatively hyperbolic groups and embedding
  theorems.
\newblock {\em Ann. of Math. (2)}, 172(1):1--39, 2010.

\bibitem[Rob96]{Rob}
Derek J.~S. Robinson.
\newblock {\em A course in the theory of groups}, volume~80 of {\em Graduate
  Texts in Mathematics}.
\newblock Springer-Verlag, New York, second edition, 1996.

\bibitem[Sha09]{Sha}
Aner Shalev.
\newblock Word maps, conjugacy classes, and a noncommutative {W}aring-type
  theorem.
\newblock {\em Ann. of Math. (2)}, 170(3):1383--1416, 2009.

\bibitem[VG10]{VG} E. P. Vdovin and A. A. Gal't. Strong reality of finite simple groups. 
\newblock {\em Sib. Mat. J.} 51(4):610--615, 2010.

\end{thebibliography}
\end{document}